\documentclass[11pt]{article}
\usepackage[a4paper,margin=2cm]{geometry}
\usepackage[english]{babel}
\usepackage{latexsym,amsmath,enumerate,graphics,enumerate,amsthm,tikz,hyperref,float}
\usepackage[mathscr]{euscript}
\usepackage[affil-it]{authblk}
\usepackage{enumerate,amsthm,dsfont,pstricks}
\usepackage{latexsym,amsmath,amssymb,amscd,wrapfig,graphicx}

\newtheorem{theorem}{Theorem}[section]
\newtheorem{lemma}[theorem]{Lemma}
\newtheorem{proposition}[theorem]{Proposition}
\newtheorem{corollary}[theorem]{Corollary}

\theoremstyle{definition}

\newtheorem{example}[theorem]{Example}

\theoremstyle{remark}

\newtheorem*{theorem*}{{\bf Theorem}}

\newtheorem*{assumption*}{{\bf Assumption}}

\let\phi=\varphi

\def\R{\mathbb{R}}

\def\0{\mathbf{0}}

\def\CB{B_+^\circ}

\newcommand{\comment}[1]{}
\newcommand{\norm}[1]{\left\Vert #1 \right\Vert}

\numberwithin{equation}{section}

\let\epsilon=\varepsilon

\makeatletter
\def\@maketitle{%
  \newpage
  \null
  \vskip 2em%
  \begin{center}%
  \let \footnote \thanks
    {\Large\bfseries \@title \par}%
    \vskip 1.5em%
    {\normalsize
      \lineskip .5em%
      \begin{tabular}[t]{c}%
        \@author
      \end{tabular}\par}%
    \vskip 1em%
    {\normalsize \@date}%
  \end{center}%
  \par
  \vskip 1.5em}
\makeatother

\begin{document}

\title{\sc \huge Order isomorphisms between cones of JB-algebras}

\author{Hendrik van Imhoff%
\thanks{Email: \texttt{hvanimhoff@gmail.com}}}
\affil{Mathematical Institute, Leiden University, P.O.Box 9512, 2300 RA Leiden, The Netherlands}

\author{Mark Roelands%
\thanks{Email: \texttt{mark.roelands@gmail.com}}}
\affil{School of Mathematics, Statistics \& Actuarial Science, University of Kent, Canterbury, CT2 7NX,
United Kingdom}

\maketitle
\date{}

\begin{abstract}
In this paper we completely describe the order isomorphisms between cones of atomic JBW-algebras. Moreover, we can write an atomic JBW-algebra as an algebraic direct summand of the so-called engaged and disengaged part. On the cone of the engaged part every order isomorphism is linear and the disengaged part consists only of copies of $\mathbb{R}$. Furthermore, in the setting of general JB-algebras we prove the following. If either algebra does not contain an ideal of codimension one, then every order isomorphism between their cones is linear if and only if it extends to a homeomorphism, between the cones of the atomic part of their biduals, for a suitable weak topology.
\end{abstract}



\section{Introduction}
Fundamental objects in the study of partially ordered vector spaces are the order isomorphisms between, for example, their cones or the entire spaces. By such an order isomorphism we mean an order preserving bijection with an order preserving inverse, which is not assumed to be linear. It is therefore of particular interest to be able to describe these isomorphisms and know for which partially ordered vector spaces the order isomorphisms between their cones are in fact linear. 

Questions regarding the behaviour of order isomorphisms between cones originated from Special Relativity. Minkowski space time is partially ordered by the causal relationship between events. Studying the structure of  order isomorphisms that preserve the causal ordering between past and future distinguishing space times played a central role in Relativity theory. Alexandrov and Ovi\v{c}innikova proved in \cite{AO} and Zeeman proved in \cite{Z} that order isomorphisms from the causal cone to itself are linear. 

Later in \cite{A}, Alexandrov extended his result to finite dimensional partially ordered vector spaces where every extreme ray of the cone is engaged. Here an extreme ray is considered engaged whenever it lies in the linear span of other extreme rays. A similar result was proved by Rothaus in \cite{R} for order isomorphisms mapping between the interior of the cones instead of the whole cones. Significant contributions to the study on the linearity of order isomorphisms were made by Noll and Sch\"{a}ffer in \cite{NS1}, \cite{NS2}, \cite{S1}, and \cite{S2}, where they considered the more general setting of cones in infinite dimensional spaces that equal the sum of their engaged extreme rays. More recently, in \cite{ILG} it is shown that this condition by Noll and Sch\"{a}ffer can be reduced further by only requiring that any element of the cone can be written as an infimum of suprema of positive linear combinations of engaged extreme vectors, see Theorem \ref{thrm,ilg}.

An important class of partially ordered vector spaces are the self-adjoint parts of $C$*-algebras, or more generally, JB-algebras. The self-adjoint part of a $C$*-algebra equipped with the Jordan product 
\[
x \circ y:={\textstyle \frac{1}{2}}(xy+yx) 
\]
tuns it into a JB-algebra. It is interesting that order isomorphisms in JB-algebras, which a priori only preserve the partial order, often also preserve the underlying Jordan algebraic structure of the space as well. It follows from a theorem of Kadison's in \cite{K} that a linear order isomorphism between $C$*-algebras mapping the unit to the unit is in fact is a $C$*-isomorphism, or a Jordan isomorphism between the self-adjoint parts of the $C$*-algebras. Moln\'{a}r studied order isomorphisms on the cone of positive semi-definite operators on a complex Hilbert space in \cite{M} and proved that they must be linear. The linearity of the order isomorphism, in turn, then yields a Jordan homomorphism on the self-adjoint operators. An interesting problem to solve is to classify the JB-algebras for which all order isomorphism on the cones are automatically linear. The analogue for JB-algebras of the theorem of Kadison's, proved by Isidro and Rodr\'{i}gues-Palacios in \cite{IRP}, will then also yield a Jordan isomorphism if the unit is mapped to the unit. 

Extreme vectors in the cone of a JBW-algebra (the Jordan analogue of a von Neumann algebra) correspond precisely to the minimal projections, or atoms. Since elements in the cones of atomic JBW-algebras are the supremum of positive linear combinations of orthogonal atoms, the setting of atomic JBW-algebras is adopted to investigate how the results in \cite{ILG} can be used to describe the order isomorphisms on their cones. In section 3 we show that an atomic JBW-algebra has an algebraic decomposition into a part that contains all engaged atoms and a part that contains the disengaged atoms, those atoms that are not engaged, Proposition~\ref{jbw}. We proceed to show that an order isomorphism between cones in atomic JBW-algebras is linear on the engaged part of the algebra according to this decomposition, Theorem~\ref{thrm,ajbw}. The disengaged atoms precisely correspond to the algebraic direct summands of dimension one and this allows to completely describe all the order isomorphisms between the cones of atomic JBW-algebras, Theorem~\ref{t:order isoms on atomic JBW}.

In section 4 we study how the results from section 3 can be extended to order isomorphisms between cones in general JB-algebras. The idea is to embed a JB-algebra into the atomic part of its bidual and investigate under which conditions the order isomorphism extends to an order isomorphism between the cones in the atomic part of the biduals. It turns out that if the order isomorphism extends to a $\sigma$-weak homeomorphism, then this extension is an order isomorphism, Proposition~\ref{extoiso}. Since a JB-algebra does not have any atoms in general, the existence of disengaged atoms in the bidual is related to the existence of ideals of codimension one. We conclude the section by proving a characterization for when the order isomorphisms between cones of these JB-algebras are linear, Theorem~\ref{t:main theorem jb}.

\section{Preliminaries}

\subsection{Partially ordered vector spaces}

Let $X$ be a real vector space.
A subset $C\subseteq X$ is a cone whenever $x,y\in C$ and $\lambda,\mu\geq 0$ imply $\lambda x+\mu y\in C$ and $C\cap -C=\{0\}$. The cone $C$ induces a partial ordering on $X$ by defining $x \le_C y$ if $y-x\in C$. The pair $(X,C)$ is called a partially ordered vector space whose partial order we often denote simply by $\leq$ instead of $\leq_C$ if no confusion can arise.
In this order $C$ equals exactly the set of all positive elements, i.e., $x\in C$ if and only if $0\leq x$.
We use the notation $(X,X_+)$ for a partially ordered vector space where the order is determined by the cone of positive elements $X_+$. The cone $X_+$ is called {\it generating} if every element in $X$ can be written as the difference of two positive vectors, that is, $X=X_+-X_+$. A partially ordered vector space $(X,X_+)$ is called {\it Archimedean} if for every $x,y\in X_+$ with $nx\leq y$ for all $n\in\mathbb{N}$ one has $x\leq 0$.

Let $(X,X_+)$ and $(Y,Y_+)$ be partially ordered vector spaces. For $\Omega\subseteq X$ and $\Theta\subseteq Y$ a map $f\colon \Omega\to \Theta$ is called {\it order preserving} if $x,y\in\Omega$ with $x\le y$ implies $f(x)\le f(y)$. If in addition $f$ is bijective and its inverse is order preserving as well, we call $f$ an {\it order isomorphism}. Furthermore, we say that $\Omega$ and $\Theta$ are {\it order isomorphic} in case there exists an order isomorphism $f\colon \Omega\to \Theta$. 

A set $F\subseteq X_+$ is called a {\it face} if it is convex and satisfies the property that if $tx+(1-t)y\in F$ for some $x,y\in X_+$ and $0<t<1$, then $x,y\in F$. For a non-empty subset $S\subseteq X_+$, the face generated by $S$ is the smallest face of $X_+$ that contains $S$ and is denoted by $\mathrm{face}(S)$.

A vector $x\in X_+$ is called an {\it extreme vector} if $0\leq y \leq x$ implies that $y=\lambda x$ for some $\lambda\geq 0$.
The ray spanned by a vector $x\in X_+$ is the set $R_x:=\{\lambda x\colon\lambda\geq 0\}$.
Whenever $x\in X_+$ is an extreme vector we refer to $R_x$ as an {\it extreme ray}. Note that one-dimensional faces of $X_+$ are exactly the extreme rays of the cone $X_+$. In the study of order isomorphism the extreme rays play an important role.
Namely, extreme rays can be characterized in a purely order theoretic way \cite[Proposition~1]{NS2} in Archimedean partially ordered vector spaces. Consequently, we have the following lemma.

\begin{lemma}\label{l:extreme vectors bijection}
Suppose $(X,X_+)$ and $(Y,Y_+)$ are Archimedean partially ordered vector spaces and let $f\colon X_+\to Y_+$ be an order isomorphism. Then $f$ maps the extreme rays of $X_+$ bijectively onto the extreme rays of $Y_+$.
\end{lemma}

An extreme ray $R_x$ is called an {\it engaged extreme ray} of $X_+$ if it lies in the linear span of the other extreme rays of $X_+$, that is, the extreme vector $x$ that generates the ray can be written as a linear combination of extreme vectors in $X_+\backslash R_x$. Extreme rays that are not engaged are called {\it disengaged extreme rays}. 

A subset $\Omega\subseteq X$ is called an {\it upper set} whenever $x\in \Omega$ and $x\leq y$ imply $y\in \Omega$.
Important examples of upper sets in partially ordered vector space are the whole space and the cone. For a subset $\Omega\subseteq X$ we define the {\it sup hull} of $\Omega$ as the set of elements $x\in X$ for which there is an index set $I$ and $\{x_i\colon i\in I\}\subseteq\Omega$ such that $x=\sup_{i\in I}x_i$. The {\it inf hull} of $\Omega$ is defined analogously, and the {\it inf-sup hull} of $\Omega$ is defined inductively to be the inf hull of the sup hull of $\Omega$. 

We state Theorem 3.9 from \cite{ILG} which provides a necessary condition on the cones in partially ordered vector spaces guaranteeing that all order isomorphisms between these cones are linear. Here a map $f\colon X_+\to Y_+$ is said to be linear if there is a linear map $T\colon X\to Y$ such that $T$ coincides with $f$ on $X_+$. 
\begin{theorem}\label{thrm,ilg} Suppose $(X,X_+)$ and $(Y,Y_+)$ are Archimedean partially ordered vector spaces, and let $f\colon X_+\to Y_+$ be an order isomorphism. Then $f$ is linear on the inf-sup hull of the positive linear span of the engaged extreme rays in $X_+$. 
\end{theorem}

Let $(X,X_+)$ be a partially ordered vector space.
A vector $u\in X_+$ is called an {\it order unit} whenever for all $x\in X$ there exists a $\lambda\geq 0$ such that $-\lambda u\leq x\leq \lambda u$.
The formula $$\|x\|_u=\inf\{\lambda\geq 0:-\lambda u\leq x\leq\lambda u\},$$ defines a norm whenever $(X,X_+)$ is Archimedean.
In this case, $\|\cdot\|_u$ is called the {\it order unit norm}.
A triple $(X,X_+,u)$ is an {\it order unit space} if $(X,X_+)$ is an Archimedean partially ordered vector space and $u\in X_+$ is an order unit.
The cone $X_+$ is closed with respect to the order unit norm.
We denote the interior of $X_+$ by $X_+^\circ$, and we remark that $X_+^\circ$ is an upper set.
It follows that $X_+^\circ$ consists exactly of the order units of $X_+$, and that in an order unit space the cone is generating.

\subsection{JB-algebras}

A \emph{Jordan algebra} $(A, \circ)$ is a commutative, not necessarily associative algebra such that
\[
x \circ (y \circ x^2) = (x \circ y) \circ x^2 \mbox{\quad  for all }x,y \in A.
\]
A \emph{JB-algebra} $A$ is a normed, complete real Jordan algebra satisfying,
\begin{align*}
\norm{x \circ y} &\leq \norm{x}\norm{y}, \\
\norm{x^2} &= \norm{x}^2, \\
\norm{x^2} &\leq \norm{x^2 + y^2}
\end{align*}
for all $x,y \in A$. As mentioned in the introduction, an important example of a JB-algebra is the set of self-adjoint elements of a $C^*$-algebra equipped with the Jordan product $x \circ y := \frac{1}{2}(xy + yx)$. 

The elements $x,y \in A$ are said to \emph{operator commute} if $x \circ (y \circ z) = y \circ (x \circ z)$ for all $z \in A$. An element $x\in A$ is said to be \emph{central} if it operator commutes with all elements of $A$.

For JB-algebras $A$ with algebraic unit $e$, the \emph{spectrum} of $x \in A$, $\sigma(x)$, is defined to be the set of $\lambda \in \R$ such that $x - \lambda e$ is not invertible in JB$(x,e)$, the JB-algebra generated by $x$ and $e$, see \cite[Section~3.2.3]{OS}. Furthermore, there is a continuous functional calculus: JB$(x,e) \cong C(\sigma(x))$, see \cite[Corollary~1.19]{AS}. The cone of elements with non-negative spectrum is denoted by $A_+$, and equals the set of squares by the functional calculus, and its interior $A_+^\circ$ consists of all elements with strictly positive spectrum. This cone turns $A$ into an order unit space with order unit $e$, that is,
\[ \norm{x} = \inf \{ \lambda > 0: -\lambda e \leq x \leq \lambda e \}. \]

\begin{assumption*} Every JB-algebra under consideration in the sequel is unital with unit $e$.\end{assumption*}

The \emph{Jordan triple product} $\{ \cdot, \cdot, \cdot \}$ is defined as
\[ \{x,y,z \} := (x \circ y) \circ z + (z \circ y) \circ x - (x \circ z) \circ y, \]
for $x,y,z \in A$. The linear map $U_x \colon A \to A$ defined by $U_x y := \{x,y,x\}$ will play an important role and is called the \emph{quadratic representation} of $x$. In case $x$ is invertible, it follows that $U_x$ is an automorphism of the cone $A_+$ and its inverse is $U_{x^{-1}}$ by \cite[Lemma~1.23]{AS} and \cite[Theorem~1.25]{AS}. A {\it state} $\varphi$ of $A$ is a positive linear functional on $A$ such that $\varphi(e)=1$. The set of states on $A$ is called the \emph{state space} of $A$ and is $w$*-compact by the Banach-Alaoglu theorem and therefore must have sufficiently many extreme points by the Krein-Milman theorem. These extreme points are referred to as \emph{pure states} on $A$ (cf. \cite[A~17]{AS}).

A \emph{JBW-algebra} is the Jordan analogue of a von Neumann algebra: it is a JB-algebra with unit $e$ which is monotone complete and has a separating set of normal states, or equivalently, a JB-algebra that is a dual space. A state $\varphi$ on $M$ is  said to be {\it normal} if for any bounded increasing net $(x_i)_{i\in I}$ with supremum $x$ we have $\varphi(x_i)\rightarrow\varphi(x)$. The linear space of normal states on $M$ is called the {\it normal state space} of $M$. The topology on $M$ defined by the duality of $M$ and the normal state space of $M$ is called the {\em $\sigma$-weak topology}. That is, we say a net $(x_i)_{i\in I}$ converges $\sigma$-weakly to $x$ if $\varphi(x_i)\to \varphi(x)$ for all normal states $\varphi$ on $M$. The Jordan multiplication on a JBW-algebra is separately $\sigma$-weakly continuous in each variable and jointly $\sigma$-weakly continuous on bounded sets by \cite[Proposition~2.4]{AS} and \cite[Proposition~2.5]{AS}. Furthermore, for any $x$ the corresponding quadratic representation $U_x$ is $\sigma$-weakly continuous by \cite[Proposition~2.4]{AS}. If $A$ is a JB-algebra, then one can extend the Jordan product to its bidual $A^{**}$ turning $A^{**}$ into a JBW-algebra, see \cite[Corollary~2.50]{AS}. In JBW-algebras the spectral theorem \cite[Theorem~2.20]{AS} holds, which implies in particular that the linear span of projections is norm dense, see \cite[Proposition~4.2.3]{OS}. 

An element $p\in M$ is a {\it projection} if $p^2=p$. For a projection $p\in M$ the {\it orthogonal complement}, $e-p$, will be denoted by $p^\perp$ and a projection $q$ is {\it orthogonal} to $p$ precisely when $q\le p^\perp$, see \cite[Proposition~2.18]{AS}. The collection of projections forms a complete orthomodular lattice by \cite[Proposition~2.25]{AS}, which means in particular that every set of projections has a supremum. We remark that these sets of projections need not have a supremum in $M$. 

Any central projection $c$ decomposes the JBW-algebra $M$ as a direct sum of JBW-subalgebras such that $M=U_cM\oplus U_{c^\perp} M$, see \cite[Proposition~2.41]{AS}. A minimal non-zero projection is called an \emph{atom} and a JBW-algebra in which every non-zero projection dominates an atom is called \emph{atomic}. Furthermore, by \cite[Lemma~5.58]{AS} we have that the normal state space of an atomic JBW-algebra is the closed convex hull of the set of pure states of $M$, where a normal state $\varphi$ is considered {\it pure} whenever there exists an atom $p\in M$ such that $\varphi(p)=1$.

A standard application of Zorn's lemma shows that in an atomic JBW-algebra $M$ every non-zero projection $q$ dominates a maximal set of pairwise disjoint atoms $\mathcal{P}$. If we denote the finite subsets of such a maximal set by $\mathcal{F}$, it follows that $\mathcal{F}$ is directed by set inclusion and we obtain an increasing net $(p_F)_{F\in\mathcal{F}}$ where $p_F:=\sum_{p\in F}p$ for all $F\in\mathcal{F}$. This net has a least upper bound in $M$ since the normal states determine the order on $M$ by \cite[Corollary~2.17]{AS} and in fact

\begin{equation*}
\sup \{p_F\colon F\in\mathcal{F}\}=q.
\end{equation*} 
By \cite[Proposition~2.25]{AS} and \cite[Proposition~2.5]{AS} this net converges $\sigma$-weakly to a projection, say $r$. Suppose that there is an atom $s\le q-r$. Then $s$ and $r$ are orthogonal and so $s$ is orthogonal to all atoms $p$ where $\{p\}\in\mathcal{F}$, contradicting the maximality of $\mathcal{P}$. Hence $r=q$. This is a standard argument in the theory of JBW-algebras even without the presence of atoms. Nevertheless, the following lemma is useful in our study of atomic JBW-algebras, and is therefore recorded for future reference.

\begin{lemma}\label{l:net of projections}
Let $M$ be an atomic JBW-algebra and let $q\in M$ be a non-zero projection. Then there exists a maximal set $\mathcal{P}$ of pairwise disjoint atoms dominated by $q$, and the increasing net $(p_F)_{F\in\mathcal{F}}$ indexed by the finite subsets $\mathcal{F}$ of $\mathcal{P}$ such that $p_F:=\sum_{p\in F}p$ for all $F\in\mathcal{F}$ converges $\sigma$-weakly to its least upper bound $q$. 
\end{lemma}

\subsection{Linear order isomorphisms on JB-algebras}
The linear order isomorphisms between JB-algebras have been classified, which we will use to give a complete description of order isomorphisms between cones in atomic JBW-algebras. An important result to obtain the classification of the linear order isomorphisms is \cite[Theorem~1.4]{IRP}, which we state for the convenience of the reader. A {\em symmetry} is an element $s$ satisfying $s^2 = e$. Note that $s$ is a symmetry if and only if $p := (s+e)/2$ is a projection, and $s = p - p^\perp$.

\begin{theorem}[Isidro, Rodr{\'{\i}}guez-Palacios]\label{t:jbisoms}
Let $A$ and $B$ be JB-algebras. The bijective linear isometries from $A$ onto $B$ are the mappings of the form $x \mapsto s Jx$, where $s$ is a central symmetry in $B$ and $J \colon A \to B$ a Jordan isomorphism.
\end{theorem}

This theorem uses the fact that a bijective unital linear isometry between JB-algebras is a Jordan isomorphism, see \cite[Theorem~4]{WY}. We use this property of linear isometries to prove the following corollary.

\begin{corollary}\label{orderisoms}
Let $A$ and $B$ be order unit spaces, and $T \colon A \to B$ be a unital linear bijection. Then $T$ is an isometry if and only if $T$ is an order isomorphism. Moreover, if $A$ and $B$ are JB-algebras, then these statements are equivalent to $T$ being a Jordan isomorphism.
\end{corollary}

\begin{proof}
Suppose $T$ is an isometry, and let $x \in A_+$, $\norm{x} \leq 1$. Then $\norm{e-x} \leq 1$, and so $\norm{e - Tx} \leq 1$, showing that $Tx$ is positive. So $T$ is a positive map, and by the same argument $T^{-1}$ is a positive map. (This argument is taken from the first part of \cite[Theorem 4]{WY}.)

Conversely, if $T$ is an order isomorphism, then $-\lambda e \leq x \leq \lambda e$ if and only if $-\lambda e \leq Tx \leq \lambda e$, and so $T$ is an isometry.

Now suppose that $A$ and $B$ are JB-algebras. If $T$ is an isometry, then $T$ is a Jordan isomorphism by \cite[Theorem 4]{WY}. Conversely, if $T$ is a Jordan isomorphism, then $T$ preserves the spectrum, and then also the norm since $\norm{x} = \max |\sigma(x)|$.
\end{proof}

\begin{proposition}\label{p:order_isomorphism}
Let $A$ and $B$ be JB-algebras. A map $T \colon A \to B$ is a linear order isomorphism if and only if $T$ is of the form $T = U_y J$, where $y \in \CB$ and $J\colon A\to B$ is a Jordan isomorphism. Moreover, this decomposition is unique and $y=(Te)^{1/2}$.
\end{proposition}
\begin{proof}
If $T$ is of the above form, then $T$ is an order isomorphism as a composition of two order isomorphisms. Conversely, if $T$ is an order isomorphism, then $T = U_{(Te)^{1/2}} U_{(Te)^{-1/2}} T$, and by the above corollary $U_{(Te)^{-1/2}} T$ is a Jordan isomorphism.

For the uniqueness, if $T = U_y J$, then $Te = U_y Je = U_y e = y^2$ which forces $y = (Te)^{1/2}$. This implies that $J = U_{(Te)^{-1/2}} T$, so $J$ is also unique.
\end{proof}

\section{Order isomorphisms on atomic JBW-algebras}

In this section we give a complete description of order isomorphisms between cones in atomic JBW-algebras. Furthermore, we characterize for which atomic JBW-algebras $M$ and $N$ all order isomorphisms $f\colon M_+\to N_+$ are linear. 

The class of atomic JBW-algebras provides a natural setting for Theorem~\ref{thrm,ilg}. Indeed, we proceed by describing the relation between the order theoretical notions stated in Theorem~\ref{thrm,ilg} with the atomic structure of the JBW-algebra. More precisely, in an atomic JBW-algebra the extreme vectors of the cone correspond to multiples of atoms, the disengaged atoms are precisely the central atoms, and the cone is the sup-hull of the positive linear span of the atoms.

\begin{lemma}\label{eng} In a JBW-algebra the atoms are precisely the normalized extreme vectors of the cone. \end{lemma}

\begin{proof} 
Let $M$ be a JBW-algebra. If $x\in M_+$ is a normalized extremal vector, then $x$ lies in the boundary of $M_+$, so $0\in\sigma(x)$. Suppose that there are two distinct non-zero $s,t\in\sigma(x)$. Then an application of Urysohn's lemma yields a non-zero positive function $f\in C(\sigma(x))$ such that $x\pm f\in M_+$ by the continuous functional calculus. This contradicts the extremality of $x\in M_+$, so $\sigma(x)=\{0,1\}$ since $\|x\|=1$. Hence $x$ is a projection. Again, by the extremality of $x$, this projection must be minimal, or equivalently, it is an atom. 

Conversely, if $M$ is a JBW-algebra and $p\in M$ is an atom. Then by \cite[Lemma~3.29]{AS} we have $\mathbb{R}p=U_pM$ and  \cite[Proposition~2.32]{AS} in turn implies $\mathrm{face}(p)=\mathbb{R}_+p$ from which we conclude that $p$ is an extremal vector.
\end{proof}
Note that the first part of the proof of Lemma \ref{eng} is also valid in general JB-algebras, and hence any normalized extreme vector in the cone of a JB-algebra is a minimal projection. It follows from Lemma~\ref{l:extreme vectors bijection} that an order isomorphism between cones in JBW-algebras must map the rays corresponding to atoms bijectively onto each other.

\subsection{Engaged and disengaged parts of atomic JBW-algebras}
An atomic JBW-algebra $M$ can be decomposed as a direct sum $M=M_D\oplus M_E$, where $M_D$ and $M_E$ are atomic JBW-algebras that contain all disengaged and engaged atoms of $M$, respectively. In this case, the cone $M_+$ is the direct product $M_+=(M_E)_+\times (M_D)_+$, and $(M_E)_+$ equals the sup hull of the positive linear span of the engaged atoms of $M$, which is of interest to us in light of Theorem \ref{thrm,ilg}. To carry out the construction of this decomposition we characterize the disengaged atoms in an atomic JBW-algebra.
\begin{lemma}\label{dort} Let $M$ be an atomic JBW-algebra and $p\in M$ be an atom. The following are equivalent:
\begin{enumerate}[(i)]
\item $p$ is disengaged;
\item $p$ is orthogonal to all other atoms;
\item $p$ is central.\end{enumerate}
\end{lemma}

\begin{proof} 
Let $p$ be a disengaged atom and let $q$ be an atom distinct from $p$. By \cite[Lemma~3.53]{AS} the sum $p+q$ can be written as an orthogonal sum of atoms $p+q=\sum_{i=1}^n \lambda_i q_i$. Suppose $p=q_j$ for some $j\in\{1,\ldots,n\}$. If $\lambda_j=1$ holds, then $q=\sum_{i\neq j}\lambda_i q_i$. Since $p$ equals $q_j$, it is orthogonal to all other $q_i$. Hence $p\circ q =0$ and thus $p$ and $q$ are orthogonal by \cite[Proposition~2.18]{AS}. If instead $\lambda_j\neq 1$ holds, then $p$ can be written as the non-trivial linear combination 
\[
p=\frac{1}{1-\lambda_j}\left(-q+\sum_{i\neq j}\lambda_i q_i\right)
\] 
of atoms different from $p$ which contradicts the assumption that $p$ is disengaged. The last case to consider is that $p$ does not equal any $q_i$, in which case $p=-q+\sum_{i=1}^n \lambda_i q_i$ contradicting that $p$ is disengaged yet again.
We conclude that $(i)$ implies $(ii)$.

Suppose $p$ is orthogonal to all other atoms.
In particular, $p$ operator commutes with all other atoms.
Let $q\in M$ be a projection. Let $(q_F)_{F\in \mathcal{F}}$ be a net directed by the finite subsets of a maximal set of pairwise orthogonal atoms dominated by $q$ as in Lemma~\ref{l:net of projections}. As multiplication is separately $\sigma$-weakly continuous and $p$ operator commutes with the finite sums $q_F$, the relation $p\circ(q\circ x)=q\circ(p\circ x)$ holds for all $x\in M$.
Hence, $p$ is a central projection by \cite[Lemma 4.2.5]{OS} showing that $(ii)$ implies $(iii)$.

Lastly, suppose that $p$ is central. Then $U_pM=\mathbb{R}p$ by \cite[Lemma~3.29]{AS} and we get $M=\mathbb{R}p\oplus U_{p^\perp} M$. It follows that $p$ is disengaged, showing that $(iii)$ implies $(i)$.
\end{proof}

\begin{lemma}\label{l:sup-hull}
The cone $M_+$ of an atomic JBW-algebra $M$ is the sup hull of the positive linear span of its atoms.
\end{lemma}

\begin{proof}
Let $x\in M_+$. By the spectral theorem \cite[Theorem~2.20]{AS} there exists an increasing net $(x_i)_{i\in I}$ in $M_+$ consisting of positive linear combinations of orthogonal projections
\[
x_i:=\sum_{k=1}^{n_i}\lambda_{i,k} p_{i,k}
\]
for all $i\in I$ that converges in norm to its supremum $x$. For each $i\in I$ and $1\le k\le n_i$ there also exist increasing nets $(p_{F_{i,k}})_{F_{i,k}\in\mathcal{F}_{i,k}}$ of finite sums of orthogonal atoms that converge $\sigma$-weakly to their supremum $p_{i,k}$ as in Lemma~\ref{l:net of projections}. But now the set
\[
\left\{\sum_{k=1}^{n_i}\lambda_{i,k}p_{F_{i,k}}\colon i\in I,\ F_{i,k}\in\mathcal{F}_{i,k}\right\}
\]
is in the positive linear span of the atoms in $M$ and has supremum $x$. Indeed, this set has upper bound $x$ and if $y\in M_+$ is an upper bound for each of these positive linear combinations of atoms, then $x_i\le y$ for all $i\in I$, as $M_+$ is $\sigma$-weakly closed. Hence $x\le y$ and $x$ is therefore in the sup hull of the positive linear span of the atoms in $M$. 
\end{proof}

Our next goal is to construct a central projection that dominates all disengaged atoms and its orthogonal complement dominates all engaged atoms. To that end, we define
\[
\mathcal{D}_M:=\left\{p\in M\colon \mbox{$p$ is a disengaged atom}\right\}
\]
and let $p_D$ be the supremum of $\mathcal{D}_M$ in the lattice of projections of $M$. The projection $p_D$ has the desired properties as shown in the following proposition.

\begin{proposition}\label{central} 
Let $M$ be an atomic JBW-algebra. Then $p_D$ is a central projection and any atom of $M$ is either dominated by $p_D$ or is orthogonal to $p_D$.
\end{proposition}

\begin{proof}
Let $p$ be an engaged atom. Then $p$ is orthogonal to all disengaged atoms by Lemma~\ref{dort} and therefore $q\le p^\perp$ for all $q\in\mathcal{D}_M$. Thus $p_D\le p^\perp$, or equivalently $p$ is orthogonal to $p_D$. On the other hand, if $p$ is an atom orthogonal to $p_D$, then it cannot be disengaged. It follows that every atom of $M$ is either dominated by $p_D$ or is orthogonal to $p_D$.
 
Let $p\in M$ be a projection and $(p_F)_{F\in \mathcal{F}}$ be an increasing net that converges $\sigma$-weakly to $p$ consisting of finite sums induced by a maximal set of pairwise orthogonal atoms dominated by $p$ as in Lemma~\ref{l:net of projections}. Since any atom in $M$ is either dominated by $p_D$ or orthogonal to $p_D$, it follows from \cite[Proposition~2.18, Proposition~2.26]{AS} that $p_D$ operator commutes with all atoms in $M$. Hence $p_D$ operator commutes with $p_F$ for all $F\in\mathcal{F}$ and as multiplication is separately $\sigma$-weakly continuous, $p_D$ operator commutes with $p$. We conclude that $p_D$ operator commutes with all elements in $M$ by \cite[Lemma~4.2.5]{OS}, hence $p_D$ is a central projection.
\end{proof}

The central projection $p_D$ and its orthogonal complement $p_E:=p_D^\perp$ now decompose the atomic JBW-algebra $M$ as a direct sum of JBW-algebras $M=U_{p_D}M\oplus U_{p_E} M$. We refer to $M_D:=U_{p_D}M$ as the {\it disengaged} part of $M$ and $M_E:=U_{p_E} M$ as the {\it engaged} part of $M$. We proceed to show that the disengaged part of $M$ is a sum of copies of $\mathbb{R}$, and the cone in the engaged part of this decomposition is the sup hull of the positive span of the engaged atoms.

\begin{proposition}\label{jbw}  Let $M$ be an atomic JBW-algebra.  Then there exist JBW-algebras $M_D$ and $M_E$ such that $M=M_D\oplus M_E$ which satisfy the following properties:
\begin{enumerate}[(i)]
\item $M_D=\bigoplus_{p\in\mathcal{D}_M} \mathbb{R}p$;
\item $(M_E)_+$ equals the sup hull of the positive linear span of the engaged atoms.
\end{enumerate}\end{proposition}

\begin{proof}
Decompose $M$ into its disengaged and engaged part $M=M_D\oplus M_E$. By Lemma~\ref{dort} all disengaged atoms are central projections in $M_D$, so we can write $$M_D=\bigoplus_{p\in\mathcal{D}_M} U_pM=\bigoplus_{p\in\mathcal{D}_M} \mathbb{R}p,$$
as $U_p M$ is one-dimensional by \cite[Lemma~3.29]{AS}. Since $M_E$ is an atomic JBW-algebra with unit $p_E$ by \cite[Proposition~2.9]{AS}, the second statement follows from Lemma~\ref{l:sup-hull} as the atoms in $M_E$ are precisely the engaged atoms of $M$ by Proposition \ref{central}.
\end{proof}

\subsection{Describing the order isomorphisms}
Using Proposition~\ref{jbw} we can now characterize the atomic JBW-algebras $M$ and $N$ for which every order isomorphism $f\colon M_+\to N_+$ is linear. 

\begin{theorem}\label{thrm,ajbw} 
Let $M$ and $N$ be atomic JBW-algebras with order isomorphic cones. Then any order isomorphism $f\colon M_+ \rightarrow N_+$ is linear on $(M_E)_+$, and $f[(M_E)_+]=(N_E)_+$, where $M_E$ and $N_E$ are the engaged parts of $M$ and $N$, respectively. In particular, any order isomorphism $f\colon M_+\to N_+$ is linear if and only if $M$ does not contain any central atoms.
\end{theorem}

\begin{proof} 
Proposition~\ref{jbw}$(ii)$ in conjunction with Theorem~\ref{thrm,ilg} yields that $f$ is linear on $(M_E)_+$. Consequently, the rays corresponding to the engaged atoms of $M$ must be mapped bijectively to the rays corresponding to the engaged atoms of $N$. In particular, the order isomorphism $f$ maps $(M_E)_+$ into $(N_E)_+$ since these cones are the sup hull of the positive linear span of the engaged atoms by Proposition~\ref{jbw}. Applying a similar argument to $f^{-1}$ yields $f[(M_E)_+]=(N_E)_+$. 

For the second part of the statement, suppose that $M$ does not contain any central atoms. Then $M=M_E$ by Proposition~\ref{jbw} and so $M_+=(M_E)_+$. In particular, all order isomorphisms $f\colon M_+\to N_+$ must be linear. Conversely, suppose that $M$ does contain central atoms and all order isomorphism $f\colon M_+\to N_+$ are linear. Let $f\colon M_+\to N_+$ be a linear order isomorphism. Using the notation of Proposition~\ref{jbw}$(i)$, define the map $g$ on $(M_D)_+$ by 
\[
g((x_pp)_{p\in\mathcal{D}_M}):=(x_p^2p)_{p\in\mathcal{D}_M}.
\]
Note that $g$ is a non-linear order isomorphism and therefore, by Proposition~\ref{jbw} the map $g\oplus\mathrm{Id}$ defined on $(M_D)_+\times (M_E)_+=M_+$ is a non-linear order isomorphism. It follows that $f\circ (g\oplus\mathrm{Id})\colon M_+\to N_+$ is not linear either, which yields the required contradiction.
\end{proof}

A statement similar to Theorem~\ref{thrm,ajbw} is also valid when order isomorphisms $f\colon\Omega\to\Theta$ between upper sets $\Omega\subseteq M$ and $\Theta\subseteq N$ are considered instead of cones. Note that upper sets need not be convex. For example the union of two translations of the cone is an upper set. In this case we conclude that the order isomorphism is affine instead of linear as $f(0)=0$ is no longer automatic. By an affine map $T\colon M\to N$ we mean a translation of a linear map, that is, there is a linear map $S\colon M\to N$ and an $x\in N$ such that $T=S+x$. Consequently, we say that an order isomorphism $f\colon \Omega\to \Theta$ between upper sets is affine if there exists an affine map $T\colon M\to N$ such that $T$ restricted to $\Omega$ coincides with $f$.

\begin{theorem}\label{main_cone} Let $M$ and $N$ be atomic JBW-algebras such that $M_+=(M_E)_+$, then every order isomorphism $f\colon\Omega\rightarrow \Theta$ between upper sets $\Omega\subseteq M$ and $\Theta\subseteq N$ is affine.
\end{theorem}

\begin{proof} 
Suppose that $M_+=(M_E)_+$ and let $f \colon \Omega \to \Theta$ be an order isomorphism. For any $x\in\Omega$ we have that $f$ restricts to an order isomorphism from $x+M_+$ onto $f(x)+N_+$. Define the map $\hat{f}\colon M_+\to N_+$ by 
\[
\hat{f}(y):=f(x+y)-f(x).
\]
It follows that $\hat{f}$ is an order isomorphism as it is the composition of two translations and the restriction of $f$. By Theorem~\ref{thrm,ajbw}, $\hat{f}$ must be linear and therefore the restriction of $f$ to $x+M_+$ must be affine. Hence there exists an affine map $g\colon M\rightarrow N$ that coincides with $f$ on $x+M_+$. We proceed to show that $f$ coincides with $g$ on all of $\Omega$. To that end, let $y\in\Omega$. Analogously, there is an affine map $h\colon M\to N$ that coincides with $f$ on $y+M_+$. Since $M$ is an order unit space, there exists a $z\in M$ such that $x,y\le z$ and we have $z\in\Omega$ as $\Omega$ is an upper set. It follows that $z+M_+\subseteq (x+M_+)\cap(y+M_+)$, so $g$ and $h$ coincide on $z+M_+$. Note that we can write $y$ as the affine combination
\[
y=z-(z-y)=-(z+(z-y))+2z
\]
of the elements $z+(z-y),z\in z+M_+$. We find that

\begin{align*}
f(y) & = h(y) = h(-(z+(z-y))+2z) = -h(z+(z-y))+2h(z) = -g(z+(z-y))+2g(z)\\ &= g(-(z+(z-y))+2z) = g(y),
\end{align*}
and we conclude that $f$ coincides with $g$ on $\Omega$.
\end{proof}

Next, we will completely describe the order isomorphisms $f\colon M_+\to N_+$ between the cones of atomic JBW-algebras in the following theorem. Using the notation of Proposition~\ref{jbw}, we denote by $M_E$ and $N_E$ the engaged parts of $M$ and $N$ respectively, and similarly, $M_D$ and $N_D$ are the corresponding disengaged parts. Furthermore, by $\mathcal{D}_M$ and $\mathcal{D}_N$ we denote the collection of the disengaged atoms in $M$ and $N$ respectively.

\begin{theorem}\label{t:order isoms on atomic JBW}
Let $M$ and $N$ be atomic JBW-algebras and let $f\colon M_+\to N_+$ be an order isomorphism. Then there exist $y\in N_+^\circ$, order isomorphisms $f_p\colon\mathbb{R}_+\to \mathbb{R}_+$ for all $p\in\mathcal{D}_M$, a bijection $\sigma\colon\mathcal{D}_M\to\mathcal{D}_N$, and a Jordan isomorphism $J\colon M_E\to N_E$, such that for all $x=x_D + x_E\in M_+$ with $x_D=(x_pp)_{p\in\mathcal{D}_M}$ we have 
\[
f(x) = (f_p(x_p)\sigma(p))_{p\in\mathcal{D}_M} + U_yJx_E.
\]
\end{theorem}

\begin{proof}
Let $f\colon M_+\to N_+$ be an order isomorphism. By Proposition~\ref{jbw} we can decompose $M_+$ and $N_+$ as $M_+=(M_D)_+\times (M_E)_+$ and $N_+=(N_D)_+\times (N_E)_+$. By Theorem~\ref{thrm,ajbw} we have $f[(M_E)_+]=(N_E)_+$, and therefore also $f[(M_D)_+]=(N_D)_+$. Furthermore, the rays corresponding to the disengaged atoms of $M$ are mapped bijectively to the rays corresponding to the disengaged atoms of $N$. In particular, there exists a bijection $\sigma\colon \mathcal{D}_M\to\mathcal{D}_N$ and for each $p\in\mathcal{D}_M$ there is an order isomorphism $f_p\colon\mathbb{R}_+\to \mathbb{R}_+$ such that $f(\lambda p)=f_p(\lambda)\sigma(p)$. 

For $x\in M_+$ we have $x=x_D + x_E=\sup\{x_D,x_E\}$ and we find that
\[
f(x_D + x_E)=f(\sup\{x_D,x_E\})=\sup\{f(x_D),f(x_E)\}= f(x_D)+f(x_E),
\]
where the last equality is due to $f(x_D)\in (N_D)_+$ and $f(x_E)\in (N_E)_+$. This shows that $f$ decomposes as the sum of order isomorphisms $f_D \colon (M_D)_+ \to (N_D)_+$ and $f_E \colon (M_E)_+ \to (N_E)_+$, by defining $f_D(x_D)=f((x_D,0))$ and $f_E(x_E)=f((0,x_E))$. Every $x_D\in (M_D)_+$ is of the form $x_D=(x_pp)_{p\in\mathcal{D}_M}$ and satisfies $x_D=\sup\{x_pp\colon p\in\mathcal{D}_M\}$, hence
\begin{align*}
f_D(x_D)&=f((x_pp)_{p\in\mathcal{D}_M})=f(\sup\{x_pp\colon p\in\mathcal{D}_M\})=\sup\{f(x_pp)\colon p\in\mathcal{D}_M\}\\&=\sup\{f_p(x_p)\sigma(p)\colon p\in\mathcal{D}_M\}\\&=(f_p(x_p)\sigma(p))_{p\in\mathcal{D}_M}.
\end{align*}
Moreover, since $f_E$ is a linear order isomorphism, it follows that  $f(x_E)=U_y Jx_E$ for an element $y\in N_+^\circ$ and a Jordan isomorphism $J\colon M_E\to N_E$ by Proposition~\ref{p:order_isomorphism}. 
\end{proof}

An interesting and immediate consequence of Theorem~\ref{t:order isoms on atomic JBW} is the following corollary.

\begin{corollary}
Let $M$ and $N$ be atomic JBW-algebras. Then the cones $M_+$ and $N_+$ are order isomorphic if and only if $M$ and $N$ are Jordan isomorphic.
\end{corollary}
\begin{proof}
Suppose that $M_+$ and $N_+$ are order isomorphic, and let $f\colon M_+\to N_+$ be an order isomorphism. By Theorem~\ref{t:order isoms on atomic JBW} there is a bijection $\sigma\colon \mathcal{D}_M\to \mathcal{D}_N$ and a Jordan isomorphism $J\colon M_E\to N_E$. Then $G\colon M\to N$ defined for $x=(x_D,x_E)\in M$ with $x_D=(x_p p)_{p\in\mathcal{D}_M}$ by 
\[
G((x_D,x_E)):=((x_p\sigma(p))_{p\in\mathcal{D}_M},Jx_E),
\] is a Jordan isomorphism. The converse implication is immediate.
\end{proof}

\section{Order isomorphisms on JB-algebras}

The results in the previous section completely describe the order isomorphisms between cones of atomic JBW-algebras, and our goal is to investigate how these results can be used to study order isomorphisms between cones in general JB-algebras. A key observation is that any JB-algebra can be embedded isometrically, as a JB-subalgebra, into an atomic JBW-algebra, namely the atomic part of the bidual. We start by determining under which conditions an order isomorphism between cones of JB-algebras can be extended to an order isomorphism between the cones of the corresponding atomic JBW-algebras obtained via this embedding. It turns out that it is sufficient to extend to a $\sigma$-weak homeomorphism for the preduals of the atomic JBW-algebras, guaranteeing that the extension is an order isomorphism. Furthermore, by relating the ideal structure of a JB-algebra to central atoms of its bidual, we obtain an analogue of Theorem~\ref{thrm,ajbw} for cones of JB-algebras. 

\subsection{The atomic representation of a JB-algebra}

The canonical embedding of a JB-algebra $A$ into its bidual $\string^ \colon A\hookrightarrow A^{**}$ is not only an isometry, but also extends the product of $A$ to $A^{**}$ by \cite[Corollary~ 2.50]{AS}. Furthermore, let $z$ be the central projection in $A^{**}$ as in \cite[Lemma~3.42]{AS} such that 
\[
A^{**}=U_zA^{**}\oplus U_{z^\perp}A^{**}
\]
where $U_zA^{**}$ is atomic and $U_{z^{\perp}}A^{**}$ is purely non-atomic. In the sequel we will denote the atomic part $U_zA^{**}$ of $A^{**}$ by $A^{**}_a$. The quadratic representation $U_z\colon A^{**}\to A^{**}_a$ corresponding to the central projection $z$ defines a surjective Jordan homomorphism by \cite[Proposition~2.41]{AS}. Hence we obtain a Jordan homomorphism $U_z\circ \string^\colon A\to A^{**}_a$. It is a standard result for $C$*-algebras that the composition of the canonical embedding $\string^$ and the multiplication by $z$ is an isometric algebra embedding, see for example the preliminaries in \cite{Ake}, and the proof for JB-algebras is the same; see \cite[Proposition~1]{FR} for a proof for JB*-triples, which are a generalization of JB-algebras. Hence we can view $A$ as a JB-subalgebra of $A^{**}_a$. 

As $A^{**}_a$ is a JBW-algebra, it is a dual space, and it follows from \cite[Corollary~2.11]{RW} that it is the dual of 
\begin{align*}\label{eq:predual}
A':=\overline{\mathrm{Span}}\{\varphi\colon \varphi\ \mbox{is a pure state on }A\}\qquad\mbox{(norm closure in }A^*).
\end{align*}
In particular, this yields $A'=U_z^*A^*$. Indeed, if $\varphi$ is a pure state on $A$, then it is a normal pure state on $A^{**}$, so there is an atom $p\in A^{**}$ such that $\varphi(p)=1$. It follows that $0\le\varphi(z^\perp)\le\varphi(p^\perp)=0$, so $\varphi(z^\perp)=0$. Thus for any $x\in A^{**}=U_zA^{**}\oplus U_{z^\perp}A^{**}$ we have $-\|x\|z^\perp\le U_{z^\perp}x\le \|x\|z^\perp$ since $U_{z^\perp}$ is order preserving and $-\|x\|e\le x\le \|x\|e$, and so
\[
\varphi(x)=\varphi(Uzx)+\varphi(U_{z^\perp}x)=U_z^*\varphi (x).
\]  
Hence $A'\subseteq U_z^*A^*$ as $U_z^*A^*$ is norm closed. Conversely, if $\varphi$ is a state on $A$, then $U_z^*\varphi$ is $\sigma$-weakly continuous on $A^{**}$. Suppose that $U_z^*\varphi\neq 0$, then $\varphi(z)^{-1}U_z^*\varphi$ is a normal state on $A^{**}$. Since this state annihilates $U_{z^\perp}A^{**}$, it defines a normal state on the atomic part of $A^{**}$ and by \cite[Lemma~5.61]{AS} it follows that $U_z^*\varphi\in A'$. As the state space of $A$ generates $A^*$, this proves the inclusion $U_{z}^*A^*\subseteq A'$.

Since the cone in $A^{**}_a$ is monotone complete, our next objective is to study how the cone of $A$ lies inside the cone of $A^{**}_a$ with respect to bounded monotone increasing and decreasing nets, respectively. To this end, we introduce the following notation. For a subset $B\subseteq (A_a^{**})_+$ we denote by $B^m$ the set where the suprema of all bounded monotone increasing nets in $B$ are adjoined. Similarly, we denote by $B_m$ the subset of $(A_a^{**})_+$ where all the infima of bounded monotone decreasing nets in $B$ are adjoined. If we obtain $(A^{**}_a)_+$ from $B$ by adjoining suprema and infima inductively in any order, but in finitely many steps, we say that $B$ is {\it finitely monotone dense} in $(A_a^{**})_+$. A consequence of a result by Pedersen \cite[Theorem~2]{Ped} is that the cone of the self-adjoint part $\mathcal{A}_{sa}$ of a $C$*-algebra is finitely monotone dense in $(\mathcal{A}^{**}_{sa})_+$. In \cite[Theorem~4.4.10]{OS} the analogue of Pedersen's theorem is given for JB-algebras, where it is shown that the cone of a JB-algebra $A$ is finitely monotone dense in the cone of its bidual. The next proposition verifies that the cone of a JB-algebra is finitely monotone dense in the cone of the atomic part of its bidual.

\begin{proposition}\label{p:finitely monotone dense}
Let $A$ be a JB-algebra. Then $A_+$ is finitely monotone dense in $(A_a^{**})_+$.
\end{proposition}

\begin{proof}
Let $A$ be a JB-algebra, canonically embedded into its bidual, and let $A_+\subseteq \Omega\subseteq A^{**}_+$. Furthermore, let $U_z$ be the Jordan homomorphism mapping $A$ into $A_a^{**}$. If $(x_i)_{i\in I}\subseteq \Omega$ is a bounded monotone increasing net with supremum $x$ in $A^{**}_+$, then the net $(U_zx_i)_{i\in I}$ is a bounded increasing net in $(A_a^{**})_+$ with supremum $y$ in $(A^{**}_a)_+$, as $A^{**}_a$ is a JBW-algebra and $U_z$ is order preserving. Since $U_z$ is the projection onto $A_a^{**}$, it follows that $U_zy=y$. For any normal state $\varphi$ on $A^{**}$ we have
\begin{align*}
\varphi(y-U_zx)&=\varphi(y-U_zx_i)+\varphi(U_zx_i-U_zx)=\varphi(U_zy-U^2_zx_i)+\varphi(U_zx_i-U_zx)\\&=U_z^*\varphi(y-U_zx_i)+U_z^*\varphi(x_i-x)\to 0
\end{align*}
since $U_z x_i\to y$ for $\sigma(A_a^{**},A')$, $x_i\to x$ for $\sigma(A^{**},A^*)$ and $U^*_z\varphi\in A'$. Hence $y=U_zx$ as the normal states separate the points of $A^{**}$. We have shown that 
\[
U_zA_+\subseteq U_z(\Omega^m)\subseteq (U_z\Omega)^m\subseteq (A_a^{**})_+
\]
and the fact that the analogous inclusions hold for $\Omega_m$ follows verbatim. Therefore, we conclude that the assertion holds as $A_+$ is finitely monotone dense in $A^{**}_+$ by \cite[Theorem~4.4.10]{OS}.
\end{proof}
The Kaplansky density theorem for JB-algebras \cite[Proposition~2.69]{AS} in conjunction with \cite[Proposition~2.68]{AS} states that the unit ball of a JB-algebra $A$, which is canonically embedded into its bidual, is $\sigma$-weakly dense in the unit ball of $A^{**}$. The unit ball of $A$ corresponds to the order interval $[-e,e]$ as it is an order unit space, so by applying the affine map $x\mapsto\frac{1}{2}(x+e)$, we find that consequently the unit interval $[0,e]$ of $A$ is $\sigma$-weakly dense in the unit interval $[0,e]$ of $A^{**}$. The analogue for the atomic representation also holds.

\begin{lemma}\label{s-dense}
The unit interval $[0,e]$ of a JB-algebra $A$ is $\sigma(A^{**}_a,A')$-dense in the unit interval $[0,e]$ of $A^{**}_a$.
\end{lemma}

\begin{proof}
Let $x$ be in the unit interval of $A_a^{**}$. Then $x$ lies in the unit interval of $A^{**}$ and $U_zx=x$. By the Kaplansky density theorem for JB-algebras \cite[Proposition~2.69]{AS} in conjunction with \cite[Proposition~2.68]{AS} there is a net $(x_i)_{i\in I}$ in the unit interval of $A$ that converges $\sigma$-weakly to $x$. But then the net $(U_zx_i)_{i\in I}$ lies in the unit interval of $A$ and converges $\sigma$-weakly, and therefore also for the $\sigma(A^{**}_a,A')$-topology, to $U_zx=x$.  
\end{proof}

\subsection{Extending the order isomorphism}

Let $A$ and $B$ be JB-algebras and $f\colon A_+\to B_+$ an order isomorphism. Our aim now is to extend $f$ to an order isomorphism from $(A_a^{**})_+$ onto $(B_a^{**})_+$. Since $A_+$ and $B_+$ are finitely monotone dense in $(A_a^{**})_+$ and $(B_a^{**})_+$, respectively, by Proposition~\ref{p:finitely monotone dense}, it turns out that it is sufficient to extend $f$ to a homeomorphism with respect to the $\sigma(A_a^{**},A')$-topology and the $\sigma(B_a^{**},B')$-topology.

\begin{proposition}\label{extoiso} 
Let $A$ and $B$ be JB-algebras and suppose $f\colon A_+\rightarrow B_+$ is an order isomorphism that extends to a homeomorphism $\hat{f}\colon (A^{**}_a)_+\rightarrow (B^{**}_a)_+$  with respect to the $\sigma(A_a^{**},A')$-topology and the $\sigma(B_a^{**},B')$-topology. Then the extension $\hat{f}$ is an order isomorphism.
\end{proposition}

\begin{proof} 
Let $f\colon A_+\rightarrow B_+$ be an order isomorphism and let $\hat{f}\colon (A^{**}_a)_+\rightarrow (B^{**}_a)_+$ be a homeomorphism with respect to the $\sigma(A_a^{**},A')$-topology and the $\sigma(B_a^{**},B')$-topology that extends $f$. Suppose that $A_+\subseteq \Omega\subseteq (A^{**}_a)_+$ and $B_+\subseteq \Theta\subseteq (B^{**}_a)_+$ are subsets for which $\hat{f}$ restricts to an order isomorphism from $\Omega$ onto $\Theta$. We argue that $\hat{f}$ also restricts to an order isomorphism from $\Omega^m$ onto $\Theta^m$, and from $\Omega_m$ onto $\Theta_m$. The assertion then follows as $A_+$ and $B_+$ are finitely monotone dense in $(A^{**}_a)_+$ and $(B^{**}_a)_+$ respectively, by Proposition~\ref{p:finitely monotone dense}.

We derive some useful properties of $\hat{f}$. For all $x\in \Omega^m$ we have

\begin{equation}\label{sup} 
\hat{f}(x)=\sup\left\{\hat{f}(y)\colon y\in \Omega,\;y\leq x\right\}.
\end{equation}
To see this, let $x\in\Omega^m$. We first argue that $\hat{f}(x)$ is an upper bound of $\hat{f}(y)$ for all $y\in \Omega$ with $y\leq x$. To that end, suppose $y\in\Omega$ with $y\leq x$. Remark that $x-y\in (A_a^{**})_+$. After rescaling we can apply Lemma~\ref{s-dense} to obtain a net $(y_i)_{i\in I}$ in $A_+$ that converges to $x-y$. By the continuity of $\hat{f}$, it follows that $\hat{f}(y_i+y)$ converges to $\hat{f}(x)$. By our assumption that $\hat{f}$ is order preserving on $\Omega$ we have $\hat{f}(y)\le \hat{f}(y_i+y)$ for all $i\in I$ and therefore $\hat{f}(y)\le \hat{f}(x)$ follows as $(B_a^{**})_+$ is closed. Suppose now that $z\in (B^{**}_a)_+$ is such that $\hat{f}(y)\le z$ for all $y\in\Omega$ with $y\le x$. As $x\in\Omega^m$, there is a monotone increasing net $(x_i)_{i\in I}$ in $\Omega$ with supremum $x$. Then $(x_i)_{i\in I}$ converges to $x$ by monotone completeness and so $\hat{f}(x_i)$ converges to $\hat{f}(x)$. Hence $\hat{f}(x)\leq z$, again as $(B_a^{**})_+$ is closed, showing \eqref{sup} holds. 

Secondly, for all $x\in\Omega^m$ we have 
\begin{equation}\label{inv} 
y\in \Omega\ \mbox{and }\hat{f}(y)\leq \hat{f}(x)\ \mbox{imply } y\leq x.
\end{equation}
Indeed, let $y\in\Omega$ with $\hat{f}(y)\leq \hat{f}(x)$ and $(z_i)_{i\in I}$ a net in $B_+$ that converges to $\hat{f}(x)-\hat{f}(y)$. As $\hat{f}^{-1}$ is continuous we infer $\hat{f}^{-1}(z_i+\hat{f}(y))$ converges to $x$. For all $i\in I$ we have $y=\hat{f}^{-1}(\hat{f}(y))\le\hat{f}^{-1}(z_i+\hat{f}(y))$ since $\hat{f}^{-1}$ is order preserving on $\Theta$. Then $y\leq x$ follows from the fact that $(A_a^{**})_+$ is closed. This shows \eqref{inv}. Now for all $x,y\in \Omega^m$ we have 
\begin{align*} 
x\leq y \iff &\left\{z\in \Omega\colon z\leq x\right\}\subseteq \left\{z\in \Omega\colon z\leq y\right\} \\ 
\iff &\left\{\hat{f}(z)\colon z\in \Omega,\ z\le x\right\}\subseteq \left\{\hat{f}(z)\colon z\in \Omega,\ z\le y\right\} \\ 
\implies &\sup\left\{\hat{f}(z)\colon z\in \Omega,\ z\le x\right\}\le\sup\left\{\hat{f}(z)\colon z\in \Omega,\ z\le y\right\} \\ 
\implies &\hat{f}(x)\leq \hat{f}(y),
\intertext{where the last implication is due to \eqref{sup}. Conversely, by \eqref{inv} we have for all $x,y\in\Omega^m$ that}
\hat{f}(x)\leq\hat{f}(y) \implies &\{z\in\Omega\colon z\le x\}\subseteq\{z\in\Omega\colon z\le y\} \\ 
\implies &x\leq y.\end{align*}

This shows that $\hat{f}$ is an order embedding of $\Omega^m$ into $(B_a^{**})_+$, and it remains to be shown that $\hat{f}$ maps $\Omega^m$ onto $\Theta^m$. For $x\in\Omega^m$ and a monotone increasing net $(x_i)_{i\in I}$ in $\Omega$ with supremum $x$ we have that $\hat{f}(x)$ is the supremum of the monotone increasing net $(\hat{f}(x_i))_{i\in I}$ which is contained in $\Theta$, showing that $\hat{f}$ maps $\Omega^m$ into $\Theta^m$. Similarly, $\hat{f}^{-1}$ maps $\Theta^m$ into $\Omega^m$ and we conclude that $\hat{f}$ restricts to an order isomorphism from $\Omega^m$ to $\Theta^m$. Analogously, $\hat{f}$ restricts to an order isomorphism from $\Omega_m$ to $\Theta_m$ by reversing all inequalities and replacing the suprema by infima.
\end{proof}

To the best of our knowledge it is presently unknown whether every order isomorphism between cones of JB-algebras always extends to a homeomorphism between the cones of the atomic parts of their biduals. If this open question is answered in the positive, then the results in the next section characterize the JB-algebras for which all order isomorphisms between their cones are linear.

\subsection{Automatic linearity of order isomorphisms}

Provided that an order isomorphism between cones of JB-algebras extends to an order isomorphism between the cones of the corresponding atomic parts of their biduals, its linearity depends on the absence of central atoms in these biduals by Theorem~\ref{thrm,ajbw}. Therefore, it is crucial to understand for which JB-algebras this absence is guaranteed. Since the disengaged part of an atomic JBW-algebra is an associative direct summand by Proposition~\ref{jbw}, one leads to believe that the existence of central atoms in the bidual corresponds to having a non-zero associative direct summands in the original JB-algebra. This, however, is not the case, as is illustrated by the following example\footnote{The authors would like to express their gratitude to M. Wortel for suggesting this space.}.
\begin{example}
Consider the JB-algebra $C\left([0,1];\mathrm{Sym}_2(\mathbb{R})\right)$ consisting of symmetric $2\times 2$ matrices with continuous functions on $[0,1]$ as entries. Note that the dual of this JB-algebra is $M([0,1];\mathrm{Sym}_2(\mathbb{R}))$ consisting of symmetric $2\times 2$ matrices with regular Borel measures on $[0,1]$ as entries with the dual pairing
\[
\left\langle
\begin{bmatrix}
x_1 & x_3\\ x_3 & x_2
\end{bmatrix},\begin{bmatrix}
\mu_1 & \mu_3\\ \mu_3 & \mu_2
\end{bmatrix}
\right\rangle=\int_0^1x_1(t)\ d\mu_1(t)+\int_0^1x_2(t)\ d\mu_2(t)+2\int_0^1x_3(t)\ d\mu_3(t).
\]  
Define the JB-subalgebra $A$ by
\begin{align}\label{e:alg example}
A:=\left\{\begin{bmatrix}
x_1 & x_3\\ x_3 & x_2
\end{bmatrix}\in C\left([0,1];\mathrm{Sym}_2(\mathbb{R})\right)\colon x_3(t)=0\ \mbox{for all }0\le t\le \textstyle{\frac{1}{2}}\right\}.
\end{align}
Note that $A$ does not have any non-trivial direct summands as $[0,1]$ is connected and $\text{Sym}_2(\mathbb{R})$ is a factor. In particular, $A$ does not contain an associative direct summand. However, the atomic part of the bidual equals 
\[
A^{**}_a=\left\{\begin{bmatrix}
x_1 & x_3\\ x_3 & x_2
\end{bmatrix}\in \ell^\infty\left([0,1];\mathrm{Sym}_2(\mathbb{R})\right)\colon x_3(t)=0\ \mbox{for all }0\le t\le \textstyle{\frac{1}{2}}\right\}
\]
and the elements of the form 
\[
\begin{bmatrix}
\delta_t & 0\\ 0 & 0
\end{bmatrix}\quad\mbox{or}\quad \begin{bmatrix}
0 & 0\\ 0 & \delta_t
\end{bmatrix}
\]
where $\delta_t$ denotes the point mass function at $t$ for $0\le t\le \frac{1}{2}$ are central atoms in $A^{**}_a$. 
\end{example}
This example shows that an alternative condition on the JB-algebra is needed.

\begin{lemma}\label{l:codimension1}
Let $A$ be a JB-algebra. Then $A^{**}$ contains a central atom if and only if $A$ contains a norm closed ideal of codimension one.
\end{lemma}

\begin{proof}
Suppose $p\in A^{**}$ is a central atom. Then $U_p\colon A^{**}\to \mathbb{R}p$ is a $\sigma$-weakly continuous Jordan homomorphism by \cite[Proposition~2.4]{AS} and \cite[Proposition~2.41]{AS}. Hence the corresponding multiplicative functional $\varphi_p$ defined by $U_px=\varphi_p(x) p$ for all $x\in A^{**}$ is an element of $A^*$. We conclude that $\ker\varphi$ is a norm closed ideal in $A$ of codimension one.

Conversely, if $I$ is a norm closed ideal in $A$ of codimension one, then $A/I\cong\mathbb{R}$ and the corresponding quotient map $\pi\colon A\to \mathbb{R}$ extends uniquely to a normal homomorphism $\tilde{\pi}\colon A^{**}\to \mathbb{R}$ by \cite[Theorem~2.65]{AS}. Since $\ker\tilde{\pi}$ is a $\sigma$-weakly closed ideal of $A^{**}$, it follows from \cite[Proposition~2.39]{AS} that there is a central projection $p\in A^{**}$ such that $\ker\tilde{\pi}=U_pA^{**}$. But as this implies $U_{p^\perp}A^{**}\cong\mathbb{R}$, the central projection $p^\perp$ must be an atom by \cite[Lemma~3.29]{AS}. 
\end{proof}

\begin{theorem}\label{t:main theorem jb}
Let $A$ and $B$ be JB-algebras such that $A$ does not contain any norm closed ideals of codimension one and let $f\colon A_+\to B_+$ be an order isomorphism. Then $f$ is linear if and only if it extends to a homeomorphism $\hat{f}\colon (A_a^{**})_+\to (B_a^{**})_+$ with respect to the $\sigma(A_a^{**},A')$-topology and the $\sigma(B_a^{**},B')$-topology.
\end{theorem}

\begin{proof}
If $f$ is linear then there is an element $y\in B_+^\circ$ and a Jordan isomorphism $J\colon A\to B$ such that $f=U_y J$ by Proposition~\ref{p:order_isomorphism}. Since the adjoint of $U_yJ$ is an order isomorphism between the duals of $B$ and $A$, it must map $B'$ bijectively onto $A'$. 
If we denote this restriction by $(U_yJ)'$, then its adjoint $(U_yJ)'^*$ in turn, is a bounded linear bijection from $A^{**}_a$ onto $B^{**}_a$, which must be a homeomorphism with respect to the $\sigma(A_a^{**},A')$-topology and the $\sigma(B_a^{**},B')$-topology. As the points of a JB-algebra are separated by the pure states, we conclude that $(U_yJ)'^*$ is an extension of $f$.  

Conversely, suppose that $f$ extends to a homeomorphism $\hat{f}\colon (A_a^{**})_+\to (B_a^{**})_+$ with respect to the $\sigma(A_a^{**},A')$-topology and the $\sigma(B_a^{**},B')$-topology. Then $\hat{f}$ is an order isomorphism by Proposition~\ref{extoiso} and as $A_a^{**}$ does not contain any central atoms by Lemma~\ref{l:codimension1}, it must be linear by Theorem~\ref{thrm,ajbw}. 
\end{proof}

The condition in Theorem \ref{t:main theorem jb} of the JB-algebra not having any norm closed ideals of codimension one is necessary. Indeed, if we consider the JB-algebra $A$ defined in \eqref{e:alg example}, then we can define a non-linear order isomorphism on $A_+$ that does extend to a $\sigma(A^{**}_a,A')$-homeomorphism on the atomic part of its bidual. Indeed, let $\lambda\colon[0,1]\to\mathbb{R}_+$ be a non-constant strictly positive continuous map such that $\lambda(t)=1$ on $(\frac{1}{2},1]$. Define the map $f\colon A_+\to A_+$ by $f(x)(t):=x(t)^{\lambda(t)}$. Since taking a coordinate-wise strictly positive power is an order isomorphism on $\mathbb{R}_+^2$, the map $f$ defines an order isomorphism. However, $f$ is not homogeneous and therefore not linear, and $f$ extends to a $\sigma(A^{**}_a,A')$-homeomorphism $\hat{f}\colon (A^{**}_a)_+\to (A^{**}_a)_+$ by the same formula that defines $f$.


\begin{thebibliography}{9} \addcontentsline{toc}{subsection}{References}
\bibitem{A} A. D. Alexandrov, A contribution to chronogeometry. {\em Canad. J. Math.}, \textbf{19}, (1967), 1119--1128.
\bibitem{AO} A. D. Alexandrov and V. V. Ov\v{c}innikova, Notes on the foundations of relativity. {\em Proc. Am. Math. Soc.}, \textbf{11}, (1953), 95--110.
\bibitem{AS} E. M. Alfsen and F.W. Shultz, {\em Geometry of State Spaces of Operator Algebras},
Mathematics: Theory \& Applications, Birkh\"auser Boston, Inc., Boston, MA, 2003.
\bibitem{Ake} C. A. Akeman, A Gelfand representation theory for $C$*-algebras. \emph{Pacific J. Math.}, \textbf {39} (1971), 1--11.
\bibitem{FR} Y. Friedman and B. Russo, The Gelfand-Naimark theorem for JB*-triples. \emph{Duke Math. J.}, {\bf 53}(1) (1986), 139--148.
\bibitem{OS} H. Hanche-Olsen, E. St{\o}rmer, \emph{Jordan operator algebras}. Monographs and Studies in Mathematics, 21. Pitman (Advanced Publishing Program), Boston, MA, 1984.
\bibitem{ILG} O. van Gaans, H. van Imhoff and B. Lemmens, On the linearity of order-isomorphisms. Preprint: \url{https://arxiv.org/abs/1904.06393}, 2019.
\bibitem{IRP} J. M. Isidro and  A. Rodr{\'{\i}}guez-Palacios, Isometries of {${\rm JB}$}-algebras.
 {\em Manuscripta Math.} {\bf 86}(3) (1995), 337--348.
\bibitem{K} R. V. Kadison, A generalized Schwarz inequality and algebraic invariants for operator algebras. {\em Ann. of Math. } {\bf 56}(2), (1952), 494–-503. 
\bibitem{M} L. Moln\'{a}r, Order-automorphisms of the set of bounded observables. \emph{J. Math. Phys. }, {\bf 42}(12) (2001), 5904--5909.
\bibitem{WY} J. D. Maitland Wright and M.A. Youngson, On isometries of {J}ordan algebras. {\em J. London Math. Soc.}  {\bf 17}(2) (1978), 339--344.
\bibitem{NS1} W. Noll and J. J. Sch{\"a}ffer, Order, gauge, and distance in faceless linear cones; with examples relevant to continuum mechanics and relativity. {\em Arch. Rational Mech. Anal. } {\bf 66}(4) (1977), 345--377.
\bibitem{NS2} W. Noll and J. J. Sch{\"a}ffer, Order-isomorphisms in affine spaces. {\em Ann. Mat. Pura Appl. } {\bf 117} (1978), 243--262.
\bibitem{Ped} G. K. Pedersen, Monotone Closures in Operator Algebras. \emph{ Amer. J. Math.}, \textbf{94}(4), 955--962, 1972.
\bibitem{RW} M. Roelands and M. Wortel, Hilbert isometries and maximal deviation preserving maps on JB-algebras. Preprint: \url{https://arxiv.org/abs/1706.10259}.
\bibitem{R} O. S. Rothaus, Order isomorphisms of cones. {\em Proc. Am. Math. Soc.}, \textbf{17}, (1966), 1284--1288.
\bibitem{S1} J. J. Sch{\"a}ffer, Order-isomorphisms between cones of continuous functions. {\em Annali di Matematica } {\bf 119}(1) (1979), 205--230.
\bibitem{S2} J. J. Sch{\"a}ffer, Order, gauge, and distance in faceless linear cones. II. Gauge-preserving bijections are cone-isomorphisms. {\em Arch. Rational Mech. Anal. } {\bf 67}(4) (1978), 305--313.
\bibitem{Z} E. C. Zeeman, Causality Implies the Lorentz Group. {\em J. Math. Phys.}, \textbf{5}(4), (1964), 490--493.
\end{thebibliography}
\end{document}